\documentclass[review]{elsarticle}
\usepackage[left=3cm,top=2.5cm,right=3cm,bottom=2.5cm]{geometry}
\usepackage{lineno,hyperref}
\usepackage{graphicx,array,delarray}
\usepackage{subfigure}
\usepackage{latexsym}
\usepackage{amscd}
\usepackage{amsfonts, amsmath, amssymb}
\usepackage[utf8]{inputenc}
\usepackage{verbatim}
\usepackage{mathrsfs}
\usepackage{fancyhdr}
\usepackage{float}
\usepackage{palatino}
\modulolinenumbers[5]
\newcommand{\RR}{\mathbb{R}}

\newtheorem{theorem}{\textbf{Theorem}}

\newtheorem{definition}[theorem]{\textbf{Definition}}

\newtheorem{proposition}[theorem]{\textbf{Proposition}}
\newtheorem{remark}[theorem]{Remark}

\newenvironment{proof}[1][Proof]{\noindent\textbf{#1.} }{\ \rule{0.5em}{0.5em}}

\begin{document}

\begin{frontmatter}

\title{\textbf{Existence of a $T$-periodic solution for the monodomain model corresponding to an isolated ventricle due to ionic-diffusive relations}}

\author[mymainaddress]{Andr\'es Fraguela}
\author[mymainaddress]{Ra\'ul Felipe-Sosa\corref{mycorrespondingauthor}}
\cortext[mycorrespondingauthor]{Corresponding author}
\ead{rfsosa@fcfm.buap.mx}
\author[mysecondaryaddress1]{Jacques Henry}
\author[mysecondaryaddress2]{Manlio F. M\'arquez}
%%%%%%%%%%%%%%%%%%%%%%%%%%%%%%%%%%

\address[mymainaddress]{Facultad de Ciencias Físico Matemáticas-BUAP, Puebla, M\'exico.}
\address[mysecondaryaddress1]{CARMEN - Modélisation et calculs pour l'électrophysiologie \textit{cardiaque, Inria Bordeaux - Sud-Ouest}}
\address[mysecondaryaddress2]{Instituto de Cardiología "Ignacio Chavez", México}

\begin{abstract}
In this paper, we find relations between the ionic parameters and the diffusion parameters which are sufficient to ensure the existence of a periodic solution for a well-known monodomain model in a weak sense. We make use of the method of approximation of Faedo-Galerkin to prove the existence of weak periodic solutions of the monodomain model for the electrical activity of the heart assuming that it is periodically activated in its boundaries. Actually, this periodic solution has the same period of activation. Finally, we reflect on how these ionic-diffusive relations are useful to explain the pathophysiology of some rhythm disorders.
\end{abstract}


\end{frontmatter}

\section{Introduction}
The bidomain model is one of the most used systems for simulating the electrical activity of the heart. This model considers an active myocardium on a macroscopic scale by relating membrane ionic current, membrane potential, and extracellular potential, see Henriquez \cite{Henriquez} for more details. We consider that the heart occupies a volume $\Omega \subset \RR^3$ where for each point it is possible to define two electrical potentials which are functions of space and time: the intracellular potential and the extracellular potential, associated to the intracellular and extracellular domains respectively, separated by the cell membrane, see Sundnes, J., \cite{Sundnes}. Introduced in 1969 (by Schmidt \cite{Schmitt}), (Clerc \cite{Clerc}) and firstly developed in 1978 (Tung \cite{Tung}), (Miller \cite{Miller1}, I), the bidomain model was used to derive forward models, which compute the extracellular and body-surface potentials from given membrane potentials (Miller \cite{Miller1}, I), (Gulrajani \cite{Gulra1}), (Miller \cite{Miller2}, II) and (Gulrajani \cite{Gulra2}). Later, the bidomain model was used to take into account several models for membrane potential together, in order to create a  reaction-diffusion system (Barr \cite{Barr}), (Roth \cite{Roth}), which simulates the spreading of the membrane potential based on the following three premises: the membrane model, the bidomain model, and the Maxwell’s equations. In (Neu \cite{Nue}), (Ambrosio \cite{Ambrio}) and (Pennacchio \cite{Penn}), other mathematical derivations of macroscopic bidomain type models can be found, which were deducted from the microscopic properties of tissue and by using asymptotic and homogenization methods along with basic physical principles.

The monodomain model was conceived as a simplification of the bidomain model, assuming that the intracellular current conductivity is proportional to the extracellular current conductivity, see \cite{Sundnes}. This way, a simpler system of equations is built. This system is composed of one parabolic equation for the membrane potential and a system of ordinary equations for the activation variables. 

The monodomain model which, has advantages from both the mathematical analysis and computational point of view, was actually developed before the first bidomain model. However, surprisingly few papers have compared the monodomain with the bidomain results. Those presenting this comparison, are not very different from each other (Vigmond \cite{Vig}). Moreover, numerical simulations of the monodomain model have provided realistic results (Leon \cite{Leon}), (Hren \cite{Hren}), (Huiskamp \cite{Huis}), (Bernus \cite{Bernus1}), (Trudel \cite{Trudel}) and (Berenfeld \cite{Beren}). In (Potse \cite{Potse}), the impact of the monodomain model has been researched in an isolated human heart following the context of simulated propagation of the electrical activity of the heart. Results were compared with the bidomain model from a numerical point of of view. In this work, it is shown that the differences between the two mentioned models are small, even if extracellular potentials are influenced by fluid-filled cavities. All the properties of the membrane and extracellular potentials corresponding to the bidomain model have been accurately reproduced by the monodomain model. A formal derivation of the monodomain equation for the version presented here it can be found in \cite{Sundnes}.

On the other hand, from the qualitative point of view, the bidomain model turns to be complex, due to several aspects. Some of them are the following
\begin{itemize}
\item it is a nonlinear model,
\item it consists of a system of joined equations that include three different kinds: parabolic, elliptic, and ordinary differential equations. 
\item Moreover, taking into account that in the system coefficients are not necessarily smooth, probably there are no classic solutions. In this regard, results about the existence of a unique solution for the Cauchy problem, the existence of periodic solutions, and the stability of this type of solution demand the use of a weak or variational framework to study these systems of equations.  
\end{itemize}
Literature has a few references dealing with the well-posedness of the bidomain model. The most important of these reference are Colli-Franzone and Savar\'e's paper (Colli \cite{Colli}), Veneroni's technical report (Veneroni \cite{Vero}), and Y. Bourgault, Y. Coudière and C. Pierre's paper (Bourgault \cite{Yves}). In \cite{Colli}, the existence and uniqueness of global solutions in the bidomain model time are proven, although its approach applies only to particular cases of ionic models, typically of the form $f(u,w) = k(u) + \alpha w, $ $g(u,w) = \beta u + \gamma w,$ where $k \in C^1(\RR)$ satisfies $\inf_{\RR} k' > -\infty$. An ionic model having this form is the cubic-like FitzHugh-Nagumo model (Fitzhugh \cite{Fitz}), which is important for a qualitative understanding of the action potential propagation, however its applicability to myocardial excitable cells is limited (Keener \cite{Kee}), (Panfilov \cite{Pan}). Furthermore, from the techniques developed in (Colli \cite{Colli}), it is not possible to conclude the existence of solution for other simple two-variable ionic models widely used in the literature for modeling myocardial cells, such as the Aliev-Panfilov (Aliev \cite{Aliev}) and Rogers-McCulloch (Rogers \cite{McCu}) models. In Colli-Franzone and Savarés \cite{Vero}, these kind of results have been extended to a more general and realistic ionic model, namely those taking the form of Luo-Rudy model (Luo \cite{Luo}). However, this extension does not include the Aliev-Panfilov and Rogers-McCulloch models. On the other hand, in Reference \cite{Yves}, global weak solutions are obtained for ionic models which can be written as a single ODE with polynomial nonlinearities. These last ionic models include the FitzHugh-Nagumo model and other simple models that are more adapted to myocardial cells, such as the Aliev-Panfilov and Rogers-McCulloch models. In Hernandez \cite{nues}, the authors obtain results about the existence and uniqueness of weak solutions for the Cauchy problem, but the monodomain model is analyzed in an isolated ventricle, and it is activated by the Purkinje fibers. Moreover, the ionic model is the Rogers-McCulloch one. In this regard, homogeneous boundary conditions of Neumann type were considered.

In relation to the existence of periodic solutions, there are some relevant recent papers. For example, in \cite{Hieber1}, M. Hieber et. al. prove the periodic version of the Da Prato-Grisvard theorem as well as its extension to semi-linear evolution equations. That is to say, from this outcome they proved the existence of strong $T$-periodic solution of the bidomain model, that is, in the context of the corresponding abstract evolution problem. In their research, they considered $T$-periodic external sources of intra and extracellular current $I_i, I_e$. However, the method of \cite{Hieber1} is applied to just a few models describing the ionic transport like the FitzHugh-Nagumo, Aliev-Panfilov, or Rogers-McCulloch models. It is worth noting that this methodology is very general. In addition, although it is applied to concrete ionic models, the results of the existence of periodic solutions are not given in terms of relations between the parameters of the model. \textbf{In other words, the method does not include a study of the ranges of values of the model parameters for which the mathematical results are true.} These relations are very important because they explain which are the physiological causes that underlay the generation or lost of cardiac rhythm. We must say the \cite{Giga} works with the same approach.  

In our paper, we study the monodomain model in an isolated region of the heart (a ventricle), in which the $T$-periodic activation is located at the endocardium, and we consider an ionic model of Roger-McCulloch type. In this case, we use the variational formulation of the boundary problem from which is obtained a notion of ``weak" solution is obtained as it happened in \cite{Yves} and \cite{nues}. Subsequently, we prove the existence of a weak $T$-periodic solution.

There is twofold difference between the present work and previous studies. First, we consider the activation in the inner wall of the heart. In this sense, the activation is reflected in the model as a boundary condition of Neumann type. On the other hand, our method allows us to find relations between the ionic parameters and the diffusive parameters which are sufficient conditions for the existence of periodic solutions. With these relations, we are able to obtain a result that could explain some aspects related to the generation of ventricular arrhythmias, specifically with the "degeneration" of a stable rhythm (sinus rhythm) toward ventricular fibrillation. 
 
\subsection{Monodomain model for the electrical activity of the heart}
In this paper, the isolated ventricle is represented as a bounded region $\Omega \subset \RR^3,$ whose boundary $\partial \Omega$ is composed of two disjoint regions, $\Gamma_0$ and $\Gamma_1$. The component $\Gamma_0$ depicts the epicardium (in direct contact with the insulating medium), while $\Gamma_1$ represents the endocardium (where electrical activation takes place).

To describe the generation and propagation of electrical activity in a ventricle, we use the monodomain model, see \cite{Sundnes}. This model can be formulated by means of a parabolic reaction-diffusion partial differential equation with boundary conditions on $\Gamma_0$ and $\Gamma_1$. Joined with this equation, we have a system of ordinary differential equations modeling the ionic currents.

The system of equations reads as follows:
\begin{align}
\label{ecC}&C_m\frac{\partial \widehat{u}}{\partial t} + f_{ion}(\widehat{u},\widehat{w}) - \nabla\cdotp\left(\sigma(\mathbf{x})\nabla \widehat{u}\right) = 0, 
\;\;\; (t,\mathbf{x}) \in (0,\infty)\times \Omega,\\
\label{ecCI}&\frac{\partial \widehat{w}}{\partial t} = g(\widehat{u},\widehat{w}),\;\;\; \;\;\; (t,\mathbf{x}) \in (0,\infty)\times \Omega,\\ 
\label{condFP1}&\left(\sigma(\mathbf{x}) \nabla(\mathbf{x}) \widehat{u}\right)\cdotp \mathbf{n}(\mathbf{x}) = 0,\;\;\; (t,\mathbf{x}) \in (0,\infty)\times \Gamma_0,\\
\label{condFP}&\left(\sigma(\mathbf{x}) \nabla \widehat{u}\right)\cdotp \mathbf{n}(\mathbf{x}) = s(t)\varphi(\mathbf{x}),\;\;\; (t,\mathbf{x}) \in (0,\infty)\times 
\Gamma_1.
\end{align}
The equation (\ref{ecC}) describes the balance of current in the heart muscle, where:
\begin{itemize}
\item $\widehat{u}, \widehat{w}$ are the membrane potential and the activation variable, respectively,
 \item $C_m\frac{\partial \widehat{u}}{\partial t},$ is the capacitive current,
 \item $-\nabla\cdotp \left(\sigma \nabla \widehat{u}\right),$ is the ohmic current,
 \item  and $f_{ion}(\widehat{u},\widehat{w}),$ is the ionic current induced by the 
cellular activation. 
 
\end{itemize}
The equation (\ref{ecCI}) describes the dynamic of the 
activation and inhibition variables for the ionic channels.
In this model, the anisotropic characteristics of the tissue 
are given by the tensor $\sigma$. Also, $C_m = \chi c_m,$ where $c_m > 0$ is the cellular membrane capacitance, and  $\chi > 0$ denotes its area per unit of volume. The function $s: (0,\infty) \rightarrow \RR$ is $\widetilde{T}$-periodic (sinus rhythm) and it corresponds to the period activation of endocardium which is given through Purkinje fibers.  Finally, $\varphi: \Omega \rightarrow \RR,$ is the spatial density of activation and thus $s(t)\varphi(\mathbf{x})$ represents the external current source. In \eqref{condFP1} and \eqref{condFP}, $\mathbf{n}(\mathbf{x})$ represents the unitary outside normal in each point $\mathbf{x}$ of boundary. 

In this article, we consider a phenomenological model for the ionic currents, specifically, the Rogers-McCulloch model, see \cite{McCu}.
\begin{align}\label{fion}
\nonumber
f_{ion}(\widehat{u},\widehat{w}) &= a_1\left(\widehat{u} - u_{res}\right)\left(\widehat{u} - u_{th}\right)
\left(\widehat{u} - u_{peak}\right) \\
&+ a_{2}\left(\widehat{u} - u_{res}\right)\widehat{w},
\end{align}
where 
\begin{align}
a_1 = \frac{c_1}{u^2_{amp}},\;\;\;\; a_2 = \frac{c_2}{u_{amp}}
\end{align}
and 
\begin{align}\label{g}
g(\widehat{u},\widehat{w}) = b\left(\widehat{u} - u_{res} - c_3\widehat{w}\right).
\end{align}
Here, $u_{res} < 0, u_{peak} > 0$ and $u_{amp}$ are the resting value, peak value and total amplitude of the cardiac action potential, respectively, $u_{amp} = u_{peak} - u_{res} > 0$
and $u_{th} = u_{res} + au_{amp}.$ This model is the reparameterized FitzHugh-Nagumo model. 

In a more realistic context the parameters should depend on spatial variables, while here we assume that the parameters are constant. For the sake of simplicity in computations and to facilitate a proper physiological interpretation of results, we will make the following change of variables and time reparameterization:
\begin{align*}
\widehat{u} = u + u_{res},\;\;\; \widehat{w} = \xi w,
\end{align*}
and 
\begin{align*}
t = \epsilon \tau,
\end{align*}
where $\xi, \epsilon$ are auxiliary parameters that will be useful later. With this change of variables the new unknown functions are 
\begin{align*}
u(\tau,\mathbf{x}) = \widehat{u}(\epsilon\tau,\mathbf{x}) - u_{res},\;\;\; w(\tau,\mathbf{x}) = \frac{1}{\xi}\widehat{w}(\epsilon\tau,\mathbf{x}).
\end{align*}
In these new variables the boundary problem has the following form
\begin{align}
\label{ecCn}&\frac{\partial u}{\partial \tau} + \widehat{f}(u,w) - \nabla\cdotp\left(\widehat{\sigma}\nabla u\right) = 0, 
\;\;\; (\tau,\mathbf{x}) \in (0,\infty)\times \Omega,\\
\label{ecCIn}&\frac{\partial w}{\partial \tau} = \widehat{g}(u,w),\;\;\; \;\;\; (\tau,\mathbf{x}) \in (0,\infty)\times \Omega,\\ 
\label{condFP1n}&\left(\widehat{\sigma} \nabla u\right)\cdotp \mathbf{n}(\mathbf{x}) = 0,\;\;\; (\tau,\mathbf{x}) \in (0,\infty)\times \Gamma_0,\\
\label{condFPn}&\left(\widehat{\sigma} \nabla u\right)\cdotp \mathbf{n}(\mathbf{x}) = \widetilde{s}(\tau)\varphi(\mathbf{x}),\;\;\; (\tau,\mathbf{x}) \in (0,\infty)\times 
\Gamma_1,
\end{align}
where
\begin{align}
\label{fionN}\widehat{f}(u,w) &= \frac{\epsilon c_4}{C_m}u + f(u,w),\\
\label{gN}\widehat{g}(u,w) &= \epsilon b(u - \xi c_3w),
\end{align}
\begin{align}\label{fN}
f(u,w) = \frac{\epsilon}{C_m}\left(a_1u^3 + \xi a_2uw - a_1(u_{pr} + u_{tr})u^2\right),
\end{align}
$u_{pr} = u_{peak} - u_{res}, u_{tr} = u_{th} - u_{res},$ and $c_4 = a_1u_{tr}u_{pr},$ are positive constants. Also, $\widetilde{s}(\tau) = \frac{\epsilon}{C_m}s(\epsilon \tau)$ and $\widehat{\sigma} = \frac{\epsilon}{C_m}\sigma$.

Note that, if the function $s$ is $T$-periodic then $\widetilde{s}$ is $\frac{T}{\epsilon}$-periodic, and if $\widetilde{s}$ is $T$-periodic then $s$ is $\epsilon T$-periodic. The same applies to unknown functions $u, w$. In what follows, we assume that $\widetilde{s}$ is $T$-periodic where $T = \frac{\widetilde{T}}{\epsilon}$. Below, we will prove the existence of $T$-periodic solutions of problem \eqref{ecCn}-\eqref{condFPn}.

From this, we will conclude the existence of a $\widetilde{T}$-periodic solution of problem \eqref{ecC}-\eqref{condFP}.
\subsection{Some notations and assumptions}\label{sectioNota}
We denote $H = L^2(\Omega)$, the square summable functions space on $\Omega$, and $V = H^1(\Omega)$, the corresponding Sobolev space. It is well known that both are Hilbert spaces. We will also consider the Banach spaces $L^p(\Omega)$, when $2 \leq p \leq 6$. We have the Gelfand triple 
\begin{align*}
V \subset H \subset V^{*},
\end{align*}
with dense and continuous embeddings, and where $\left\{V,V^{*}\right\}$ forms an adjoint pair with duality product $\left\langle\cdot, \cdot\right\rangle_{V\times V^{*}}$ satisfying 
\begin{align*}
\left\langle v,h\right\rangle_{V\times V^{*}} = \left(v,h\right),\;\;\; \hbox{for all $v \in V, h \in H,$}
\end{align*}
here, $\left(\cdot,\cdot\right)$ represents the inner product in $H$. Observe that from the Sobolev embedding theorem, the following inclusions
$$V \subset L^p(\Omega) \subset H \subset L^{p'}(\Omega)
\subset V^{*}$$
are continuous for all $2 \leq p \leq 6.$ In particular, in this paper we will make use of these embeddings for $p = 4$ in this paper. In what follows, the next specific assumptions will be taken into account
\begin{description}
\label{H1}\item[(h1)] $\Omega$ has boundary $\partial \Omega$ of class $C^2$.
\item[(h2)] $\sigma$ is a symmetric matrix-function of the spatial variable $\mathbf{x} \in \Omega$, with entries in $C^1(\overline{\Omega})$, and such that there are positive constants $m$ and $M$ such that
\begin{align*}
0 < m|\xi|^2 \leq \xi^t \sigma(x)\xi \leq M|\xi|^2\;\;\; 
\hbox{for all $\xi \in \RR^3\setminus \{0\},$}
\end{align*}
for almost all $\mathbf{x} \in \Omega.$
\item[(h3)] $\widetilde{s} \in L^{\infty}(0,\infty),$ and $\varphi \in L^2(\Gamma_1).$ Besides, we assume that $\widetilde{s}(t)$ is $T$-periodic.
\end{description}
\section{Variational formulation, notion of weak solution} 
In this section, we give a variational formulation for the boundary value problem (\ref{ecCn})-(\ref{condFPn}). In addition, we define the notion of weak solution for this problem. From now on, we return to the notation $t$ for the time variable.

\begin{remark}\label{com1}
As it can be easily seen in (\ref{fionN}) and (\ref{gN}) that the functions $\widehat{f}$ and $\widehat{g}$ have the form $\widehat{f}(u,w) = f_1(u) + f_2(u)w$ and $\widehat{g}(u,w) = g_1(u) + g_2w$,
where $p = 4$,
\begin{align*}
|f_1(u)| \leq l_1 + l_2|u|^{p-1},\;\;\;|f_2(u)| \leq a_2|u|^{p/2 - 1},\;\;\;|g_1(u)| \leq \frac{b}{2}\left(1 + |u|^{p/2}\right),
\end{align*}
and
\begin{align*}
l_1 = \frac{a_1 \epsilon}{C}\left(\frac{1}{3}(u_{tr} + u_{pr}) + \frac{2}{3}u_{tr}u_{pr}\right),\;\;\;l_2 = \frac{a_1 \epsilon}{C}\left(1 + \frac{2}{3}(u_{tr} + u_{pr}) + \frac{1}{3}u_{tr}u_{pr}\right).
\end{align*}
Applying Lemma 25 of \cite{Yves}, we have that the maps $(u,w) \in L^p(\Omega)\times H \mapsto \widehat{f}(u,w) \in L^{p'}(\Omega),$ and $(u,w) \in L^p(\Omega)\times H \mapsto \widehat{g}(u,w) \in H,$ are well defined. 
Furthermore, one shows that 
\begin{align*}
\|\widehat{f}(u,w)\|_{L^{p'}(\Omega)} &\leq A_1|\Omega|^{1/p'} + A_2\|u\|^{p/p'}_{L^p} 
+ A_3\|w\|^{2/p'}_{H},\\
\|\widehat{g}(u,w)\|_{H} &\leq B_1|\Omega|^{1/2} + B_2\|u\|^{p/2}_{L^p}
+ B_3\|w\|_{H},
\end{align*}
where $A_i, B_i$ with $i = 1, 2, 3$ are positive constants, given by
\begin{align*}
&A_1 = l_1,\;\; A_2 = l_2 + \frac{2}{3}\xi a_2,\;\; A_3 = \frac{a_2 \xi}{3},\\
&B_1 = \frac{\epsilon b}{2},\;\; B_2 = \frac{\epsilon b}{2},\;\; B_3 = \xi c_3.
\end{align*}
Note that these constants depend on the parameters of the model, specifically those involved in $\widehat{f}(u,w)$ and $\widehat{g}(u,w)$. In our paper, we obtain a result about the existence of periodic solutions which imposes some restrictions on these parameters that can be interpreted in a physiological sense. Besides, in Section 6.3 of \cite{Yves} the existence of constants $n, \lambda, r, q > 0$ is also established, such that 
\begin{align*}
\lambda u\widehat{f}(u,w) - w\widehat{g}(u,w) \geq n|u|^{p} - r(|u|^2 + |w|^2) - q, 
\end{align*}
for all $(u,w) \in \RR^2.$
\end{remark}

Suppose now that $(u,w)$ is a classic vector solution of problem (\ref{ecCn})-(\ref{condFPn}),
and multiply (\ref{ecCn}) by a function $v \in V$ and (\ref{ecCIn}) by some $h \in H$, respectively. While integrating by parts in the previously obtained equations, we have:
\begin{align*}
&\int_{\Omega}\frac{\partial u}{\partial \tau} v + \int_{\Omega}\widehat{\sigma} \nabla u\nabla v
+ \int_{\Omega}\widehat{f}(u,w)v = \widetilde{s}(\tau)\int_{\Gamma_1}\varphi v ds,\\
&\int_{\Omega}\frac{\partial w}{\partial \tau} h - \int_{\Omega}\widehat{g}(u,w)h = 0,
\end{align*}
from which it is obtained that
\begin{align}
&\label{FormV1}\frac{\partial}{\partial \tau}\int_{\Omega}uv + \int_{\Omega}\widehat{\sigma} \nabla u\nabla v
+ \int_{\Omega}\widehat{f}(u,w)v = \widetilde{s}(\tau)\int_{\Gamma_1}\varphi v ds,\\
&\label{FormV2}\frac{\partial}{\partial \tau}\int_{\Omega}w h + \int_{\Omega}\widehat{g}(u,w)h = 0.
\end{align}
Let $\phi \in \mathcal{D}(I)$ be an infinitely differentiable function with compact support in an interval $I = (0,\hat{t})$ for $\hat{t} > 0$. Multiplying in \eqref{FormV1} and \eqref{FormV2} by $\phi$ and integrating by parts, we have
\begin{align}
&\label{FormV3}-\int_I\int_{\Omega}uv\phi' + \int_I\int_{\Omega}\widehat{\sigma} \nabla u\nabla v\phi
+ \int_I\int_{\Omega}\widehat{f}(u,w)v\phi = \int_I\int_{\Gamma_1}\widetilde{s}\varphi v\phi ds,\\
&\label{FormV4}-\int_I\int_{\Omega}w h\phi' - \int_I\int_{\Omega}\widehat{g}(u,w)h\phi = 0,
\end{align}
thus, if $(u,w) \in V\times H,$ the integrals in the left-side of (\ref{FormV3})-(\ref{FormV4}) are finite by Remark \ref{com1}. 

Define the functional $\varphi^{*} \in V^{*}$ given by
\begin{align*}
\left\langle v,\varphi^{*}\right\rangle_{V\times V^{*}} = \int_{\Gamma_1}\varphi v,\;\;\; \hbox{for $v \in V,$}
\end{align*}
which satisfies $\left\|\varphi^{*}\right\|_{V^{*}} \leq \mathcal{N}_\mathcal{T}\left\|\varphi\right\|_{L^2(\Gamma_1)}$, where $\mathcal{N}_{\mathcal{T}}$ is the norm of the trace operator from $V$ in $L^2(\Gamma_1)$. 

Before giving the definition of weak solution, we will establish some properties of the bilinear form
\begin{align}\label{biform}
a(u,v) := \int_{\Omega}\widehat{\sigma} \nabla u\nabla v + \frac{\epsilon c_4}{C}\int_{\Omega}uv,
\end{align} 
defined in $V \times V.$
\subsection{Some important results about the bilinear form $a(u,v)$}\label{epi1}
The bilinear form $a(u,v)$, defined in \eqref{biform}, will play a crucial role in our paper. Hence, it is required to establish some of the bilinear form properties which will be used below.  

Taking into account the properties of the matrix $\widehat{\sigma}$, the existence of the positive constants $\alpha$ and $\mathcal{M}$ can be proven, such that
\begin{align*}
\begin{array}{c}
\left|a(u,v)\right| \leq \mathcal{M}\left\|u\right\|_{V}\left\|v\right\|_{V},\\
\alpha\left\|u\right\|^2_{V} \leq a(u,u) + \alpha\left\|u\right\|^2_{H},
\end{array}\;\;\; \hbox{for all $u,w \in V,$}
\end{align*}
that is, $a$ is a continuous coercive bilinear form. It shows there is a linear operator $\mathcal{A}: V \subset V^{*} \rightarrow V^{*},$ which is bounded, sectorial, closed and densely defined, such that 
\begin{align*}
a(u,v) = \left\langle v, \mathcal{A}u\right\rangle_{V\times V^{*}}.
\end{align*}
Next, we consider the restriction of $\mathcal{A}$ to $H$, denoted by $A$. Concretely, the operator $A$ is defined in the following form
\begin{align*}
\left\{
\begin{array}{c}
\mathcal{D}(A) = \left\{v \in V; \mathcal{A}v \in H\right\},\\
Av = \mathcal{A}v.
\end{array}
\right.
\end{align*} 
We need more details about $A$. Since, we have assumed some previous conditions on $\sigma$ and $\Omega$, given by \textbf{(h1)} and \textbf{(h2)}, it follows that 
\begin{align*}
\left\{
\begin{array}{c}
\mathcal{D}(A) = \left\{v \in H^2(\Omega); \left.\left(\widehat{\sigma}\nabla v\right)\cdot \mathbf{n}\right|_{\partial \Omega} = 0\right\},\\
Av = -\nabla \cdot\left(\widehat{\sigma} \nabla v\right) + \frac{\epsilon c_4}{C_m} v,
\end{array}
\right.
\end{align*}
consequently, $A$ is a self-adjoint and positive operator which has a compact inverse. Then, there is an increasing sequence of eigenvalues $0 < \lambda_0 \leq \lambda_1 \leq \ldots \leq \lambda_n \leq\ldots;$ associated to a corresponding sequence of eigenvectors $\left(\psi_i\right)^{\infty}_{i = 0} \subset V$, such that 
\begin{align*}
\lim_{n \rightarrow \infty} \lambda_n = \infty,
\end{align*}
and the eigenvector sequence is an orthonormal basis for $H$. Note that
\begin{align*}
a(\psi_i,v) = \left\langle v, \mathcal{A}\psi_i\right\rangle_{V\times V^{*}} = \left(A\psi_i,v\right) = \lambda_i\left(\psi_i,v\right),
\end{align*}
in particular
\begin{align*}
a(\psi_i,\psi_i) = \lambda_i.
\end{align*}
On the other hand, since $A$ is sectorial, the family of bounded operators $e^{-tA}$ ($t\geq 0$) can be defined such that $e^{-0A} = I_d$ ($I_d$ being the identity operator in $H$). This family of operators it is said to be a semigroup and it satisfies the following properties: given $\left(\lambda_i\right)^{\infty}_{i = 0}$ and $\left(\psi_i\right)^{\infty}_{i = 0}$ as before, each operator $e^{-tA}$ has spectrum $\left(e^{-\lambda_i t}\right)^{\infty}_{i = 0}$, and $\left(\psi_i\right)^{\infty}_{i = 0}$ are the corresponding eigenvectors, for all $t>0$. Furthermore, $1 \notin \sigma(e^{-tA})$. On the other hand, $\left(I_d - e^{-TA}\right)^{-1}$ is a bounded densely defined linear operator, whose spectrum is $\left((1-e^{-\lambda_iT})^{-1}\right)^{\infty}_{i = 0}$ and the eigenvectors are $\left(\psi_i\right)^{\infty}_{i = 0}$, respectively. Note that $\left\|e^{-tA}\right\|_{\mathcal{L}(H)} \leq 1.$

Next, we give the definition of weak solutions for the variational formulation of the problem \eqref{ecCn}-\eqref{condFPn}. Below, we use of notation introduced in Section \ref{sectioNota}.

\begin{definition}[Weak solution]\label{wealsolution}
Fix $t_0 > 0$, let be $I = (0,t_0)$ and $Q_I = (0,t_0)\times \Omega$. A pair $(u,w)$ is said to be a local weak solution of \eqref{ecCn}-\eqref{condFPn} in the interval $I = (0,t_0)$ if this pair satisfies
\begin{align*}
u \in L^p(Q_I)\cap L^2(I;V),\; w \in L^2(Q_I),
\end{align*}
and it verifies the following equalities, in $\mathcal{D}'(I)$,
\begin{align}\label{Varprobm}
\begin{array}{c}
\dfrac{d}{d\tau}\left(u(\tau),v\right) + a(u(\tau),v) + \displaystyle\int_{\Omega}f(u,w)v = 
\widetilde{s}(\tau)\left\langle v, \varphi^{*}\right\rangle_{V\times V^{*}}\\\\
\dfrac{d}{d\tau}\left(w(\tau),h\right) + \displaystyle\int_{\Omega}\widehat{g}(u,w)h = 0.
\end{array}\hbox{for all $v \in V, h \in H, and$}
\end{align}

On the other hand, a pair $(u,w)$ is said to be a global weak solution if it is a local weak solution for any $t_0 > 0$.
\end{definition}

\section{Existence of a global $T$-periodic weak solution as the limit of a sequence of $T$-periodic solutions of Faedo-Galerkin systems}
In this section, we must prove the existence of a global $T$-periodic weak solution of the monodomain model given by \eqref{ecCn}-\eqref{condFPn} in the sense of Definition \ref{wealsolution}. First of all, we will obtain a sequence $\left\{(u_m,w_m)\right\}^{\infty}_{m = 0}$ of $T$-periodic functions which converges, in a suitable sense, to a global $T$-periodic weak solution of \eqref{ecCn}-\eqref{condFPn}. This sequence of functions is formed by solutions of a system of ordinary equations of dimension $2m + 2$ for each $m = 0,1,\ldots,$ called the Faedo-Galerkin system.As we will see later, the fact of this sequence evaluated in $t = 0$ being uniformly bounded is a sufficient condition for this sequence to converge. The Faedo-Galerkin system is defined as follows: let $\left(\psi_i\right)^{m}_{i = 0} \subset V$ denote the first $m+1$ eigenvectors of the operator $A$, associated to the eigenvalues $\left(\lambda_i\right)^{m}_{i = 0}$. Define the linear subspace 
\begin{align}\label{subspace1}
V_m = \hbox{span}\left\{\psi_0, \psi_1,\ldots,\psi_m\right\},
\end{align}
considered as a vector subspace of $V$ with the following equivalent norm 
 $$\sqrt{\frac{\epsilon c_4}{C}\left\|\cdot\right\|^2_{H} +\left\|\sigma\nabla\left(\cdot\right)\right\|^2_{H}},$$ and denote it by $H_m$ when is it considered as vector subspace of $H$. We also introduce
\begin{align}\label{resol1}
u_m(t) = \sum^m_{i = 0}u^{(m)}_i(t)\psi_i \in V_m,\;\;\; w_m(t) = \sum^m_{i = 
0}w^{(m)}_i(t)\psi_i \in H_m,
\end{align}
note that, for all $t$ $$\left\|u_m(t)\right\|^2_{V} = \sum^m_{i = 0}\lambda_i\left|u^{(m)}_i(t)\right|^2,$$ and $$\left\|w_m(t)\right\|^2_{H} = \sum^m_{i = 0}\left|w^{(m)}_i(t)\right|^2.$$

We recall that the weak equations in \eqref{Varprobm} are satisfied for all $v \in V$ and $h \in H$, respectively. In particular, they are fulfilled by the functions $\left\{\psi_0, \psi_1,\ldots,\psi_m\right\} \subset V \subset H$. If we also assume that $(u_m, w_m)$ is a global weak solution of the problem, then the coefficients $\left(u^{(m)}_0(t),w^{(m)}_0(t),\ldots,u^{(m)}_m(t),w^{(m)}_m(t)\right)$ satisfy the following system of ordinary differential equations:
\begin{align}\label{FGsys}
\begin{array}{c}
u^{(m)'}_i =  - \lambda_{i} u^{(m)}_i 
- \displaystyle\int_{\Omega}f(u_m,w_m)\psi_i + \widetilde{s}(t)\left\langle \psi_i,\varphi^{*}\right\rangle_{V\times V^{*}}\\\\
w^{(m)'}_i = \epsilon b(u^{(m)}_i - c_3\xi w^{(m)}_i),
\end{array}\;\;\; i = 0,\ldots,m,
\end{align}
where $\widehat{g}$ has been replaced by its explicit expression given by \eqref{gN}. This system of equations has $2m + 2$ equations and the same number of unknowns
$$(u^{(m)}_0,w^{(m)}_0\ldots,u^{(m)}_i,w^{(m)}_i,\ldots,u^{(m)}_m,w^{(m)}_m).$$ The system \eqref{FGsys} is the so-called Faedo-Galerkin system. If we fix the initial conditions 
\begin{align}\label{FGintc}
u^{(m)}_i(0) = u^{(0m)}_i,\;\;\; w^{(m)}_i(0) = w^{(0m)}_i,
\end{align}
then \eqref{FGsys}-\eqref{FGintc} is a Cauchy problem for a system of ordinary differential equations, in which the classic theory of ODEs can be applied. In the same context of \eqref{resol1}, we define the following initial conditions
\begin{align}\label{FGintc1}
u_{0m} := \sum^{m}_{i = 0}u^{(0m)}_i\psi_i,\;\;\; w_{0m} := \sum^{m}_{i = 0}w^{(0m)}_i\psi_i,
\end{align}
then, we say that a pair $(u_m,w_m)$ by means of \eqref{resol1} is a solution of the Cauchy problem associated to the Faedo-Galerkin system \eqref{FGsys} with initial conditions \eqref{FGintc1}, if the respective coefficients $\left\{u_i^{(m)}(t),w_i^{(m)}(t)\right\}^{m}_{i = 0}$ are solutions of the system of ODEs \eqref{FGsys} with the initial conditions $\left\{u_i^{(0m)},w_i^{(0m)}\right\}^{m}_{i = 0}$, being the components of this last vector, the coefficients of the pair $(u_{0m},w_{0m})$. It is well known that problem \eqref{FGsys}-\eqref{FGintc} has a unique solution $(u^{(m)}_{i},w^{(m)}_{i})^m_{i = 0}$ which is defined on a maximum interval $[0,t_m)$. Under suitable conditions, the maximum interval of existence can be infinite, that is, $t_m = \infty$.

For each $m$ we consider the following projection operator $P_m : V^{*} \rightarrow V^{*}$
\begin{align*}
P_m u = \sum^{m}_{i = 0}\left\langle \psi_i, u\right\rangle_{V\times V^{*}}\psi_i,\;\;\; u \in V^{*},
\end{align*}
which is a linear and bounded operator, in fact, given that, $P_m(V^{*}) \subset V$, we can consider $P_m$ as an operator from $V$ into $V$, given by
\begin{align*}
P_m v = \sum^{m}_{i = 0}\left(v,\psi_i\right)\psi_i,\;\;\; v \in V,
\end{align*}
then, it is possible to prove that $$\left\|P_m\right\|_{\mathcal{L}(V)} \leq 1 + \frac{\mathcal{M}}{\alpha},$$ and taking into account that the transpose $P^{T}_m$ of $\left.P_m\right|_{V}$ can be identified with $P_m: V^{*} \rightarrow V^{*}$, one gets $\left\|P_m\right\|_{\mathcal{L}(V)} = \left\|P_m\right\|_{\mathcal{L}(V^{*})}$, so $$\left\|P_m\right\|_{\mathcal{L}(V^{*})} \leq 1 + \frac{\mathcal{M}}{\alpha},$$ the proof of these results can be found in \cite{Yves}.
\begin{remark}
As it was mentioned before, our main goal is to establish the existence of a periodic solution of the monodomain model in an isolated ventricle, assuming periodic activation of its endocardium. The proof is based on three main steps. The first step is inspired in an idea developed by M. Farkas in Chapter 4 of \cite{Far}. In his book, M. Farkas considers quasilinear ODEs systems whose linear and non-linear parts are $T$-periodic, and he introduces a nonlinear operator between certain spaces of continuous functions associated with the system under study, in our case the Faedo-Galerkin systems. Afterward, it is shown that the quasilinear system has a non-trivial $T$-periodic solution if and only if the corresponding operator has a fixed point (see Theorem 4.1.3, \cite{Far}). Based on this scheme, we will build an M. Farkas operator for each Faedo-Galerkin system \eqref{FGsys}, and we apply the Farkas result.

The second step consists of showing the existence of a $T$-periodic solution for each Faedo-Galerkin system \eqref{FGsys}, which defines a sequence of $T$-periodic functions. Taking into account the first step, what we do is to prove the existence of a fixed point for the Farkas operator associated with the Faedo-Galerkin system, through the fixed point Schauder theorem. Additionally, it is seen that this sequence of $T$-periodic solutions is uniformly bounded in the norm of $H$.

In the third step, we use a theorem of \cite{nues} which is, in some way, a generalization of a central result of \cite{Yves}. Our result in this part says that, if the initial conditions \eqref{FGintc} of the Faedo-Galerkin systems \eqref{FGsys} are uniformly bounded in $H$, then the sequence of solutions of the Cauchy problems \eqref{FGsys}-\eqref{FGintc} possess a subsequence which weakly converges to a global weak solution of \eqref{ecCn}-\eqref{condFPn}. Finally, in order to prove the existence of $T$-periodic weak solution, we apply our result to the uniformly bounded sequence of $T$-periodic solutions of \eqref{FGsys}, obtained in the second step.
\end{remark} 
\subsection{The Farkas operator $\mathcal{K}_m$ associated to the Faedo-Galerkin system and the fixed point type result}\label{sect1}
In this section, we develop the two first steps. With this purpose, we follow Chapter 4 of \cite{Far} devoted to the study of the existence of $T$-periodic solutions of systems of $n$ equations of the type
\begin{align}\label{auxproblem}
U' = M(t)U + F(t,U),
\end{align}
where $M \in C(\RR;\RR^{n^2})$, $F \in C^1(\RR\times \RR^n;\RR^n)$ are $T$-periodic in $t$ and $U \in C(\RR,\RR^{n})$ denotes the vector solution. In \cite{Far}, an operator is defined under the next condition: the linear problem 
\begin{align}\label{linearsys1}
U' = M(t)U,
\end{align}
has no $T$-periodic solution apart from the trivial one. To do this, the following auxiliary problem is introduced:
\begin{align}\label{auxproblem1}
U' = M(t)U + b(t),
\end{align}
where $b \in C(\RR;\RR^{n})$ is $T$-periodic. The system \eqref{auxproblem1} has only a $T$-periodic solution given by the expression
\begin{align*}
U^{p}(t) = \int^T_0K(t,\tau)b(\tau)d\tau,
\end{align*}
where $K(t,\tau)$ is the Green matrix function given by
\begin{align*}
K(t,\tau) = \left\{
\begin{array}{cc}
\Phi(t)(I - \Phi(T))^{-1}\Phi(\tau), &\;\;\; \hbox{ for $0 \leq \tau \leq t \leq T$},\\
\Phi(t + T)(I - \Phi(T))^{-1}\Phi(\tau), &\;\;\; \hbox{ for $0 \leq t \leq \tau \leq T$},
\end{array}
\right.
\end{align*}
where $\Phi(t)$ is the fundamental matrix solution of the linear system \eqref{linearsys1} and $I$ is the identity matrix. This expression defines an operator on the subspace of $T$-periodic functions $b(t) \in C(\RR; \RR^{n})$, in such a way that to each $b$ it is associated the only $T$-periodic solution of the linear system \eqref{auxproblem1}. Finally, the operator $\mathcal{K}$ on the space of $T$-periodic functions $U(t) \in C(\RR; \RR^{n})$ is defined, which assigns the only periodic solution of \eqref{auxproblem1}, $U(t)$, when $b(t)$ is replaced by $F$, that is, 
\begin{align}\label{Faroper}
\mathcal{K}(U) = \int^T_0K(t,\tau)F(\tau,U(\tau))d\tau,
\end{align}
and observe that the fixed points of $\mathcal{K}$ are $T$-periodic solutions of \eqref{auxproblem}.

Returning to the Faedo-Galerkin system, note that it can be written in the following form:
\begin{align}\label{FGsys2}
\begin{array}{c}
u^{(m)'}_i =  - \lambda_{i} u^{(m)}_i 
- \left\langle \psi_i, f(u_m,w_m)\right\rangle_{V\times V^{*}} + \widetilde{s}(t)\left\langle \psi_i,\varphi^{*}\right\rangle_{V\times V^{*}},\\\\
w^{(m)'}_i = \epsilon b(u^{(m)}_i - \xi c_3w^{(m)}_i),
\end{array}\;\;\; i = 0,\ldots,m,
\end{align}
in which $u_m \in V, w_m \in H$ and thus $f(u_m,w_m) \in L^{p'}(\Omega) \subset V^{*}$, see Remark \ref{com1}. The system \eqref{FGsys2} can be formulated in the following matrix form
\begin{align}\label{FGmat}
U_m' = L_m U_m + F(U_m) + s(t)B_0,
\end{align}
where $$U_m(t) = \left(u^{(m)}_0(t), w^{(m)}_0(t), u^{(m)}_1(t),w^{(m)}_1(t),\ldots, u^{(m)}_m(t), w^{(m)}_m(t)\right) \in \RR^{2m + 2},$$
for all $t \geq 0$, and $L$ is a block diagonal matrix, whose blocks have the form
\begin{align*}
L_i = \left(
\begin{array}{cc}
-\lambda_i & 0\\
0 & -bc_3\epsilon\xi
\end{array}
\right),\;\;\; i = 0,\cdots,m,
\end{align*}
$F(U_m)$ is a vector field with components
\begin{align*}
F(U_m) = \left(-\left\langle \psi_0, f(u_m,w_m)\right\rangle_{V\times V^{*}},\epsilon bu_0,\ldots,-\left\langle \psi_m, f(u_m,w_m)\right\rangle_{V\times V^{*}}, \epsilon bu_m\right),
\end{align*}
and $B_0$ is the following vector
\begin{align*}
B_0 = \left(\left\langle \psi_0, \varphi^{*}\right\rangle_{V\times V^{*}},0,\ldots,\left\langle \psi_m, \varphi^{*}\right\rangle_{V\times V^{*}},0\right).
\end{align*} 
This system \eqref{FGmat} is a quasilinear one, like those considered in \eqref{auxproblem}. In our case, the matrix $M(t)=L_m$ is constant, so the corresponding fundamental matrix solution takes the form
\begin{align*}
\Phi(t) = e^{tL_m}, 
\end{align*} 
which is block diagonal matrix, with blocks of the form 
\begin{align*}
\Phi_i(t) = \left(
\begin{array}{cc}
e^{-\lambda_i t} & 0\\
0 & e^{-bc_3\epsilon \xi t}
\end{array}
\right),\;\;\; \hbox{for $i = 0,\ldots,m$}.
\end{align*}
The Green matrix function $K(t,\tau)$ is also a block diagonal matrix, whose blocks are
\begin{align*}
\left(
\begin{array}{cc}
K_i(t,\tau) & 0\\
0 & K_b(t,\tau)
\end{array}
\right),\;\;\; i = 0,\ldots,m,
\end{align*}
where 
\begin{align}\label{Ki}
K_i(t,\tau) = \left(1 - e^{-\lambda_iT}\right)^{-1}
\left\{
\begin{array}{cc}
e^{-\lambda_i(t - \tau)} & \hbox{if $0 \leq \tau \leq t \leq T,$}\\
e^{-\lambda_i(t + T - \tau)} & \hbox{if $0 \leq t < \tau \leq T,$}
\end{array}\;\;\; \hbox{for $i = 0,\ldots,m$,}
\right.\;\;\; 
\end{align}
and
\begin{align*}
K_b(t,\tau) = \left(1 - e^{-bc_3\epsilon \xi T}\right)^{-1}
\left\{
\begin{array}{cc}
e^{-bc_3\epsilon \xi(t - \tau)} & \hbox{if $0 \leq \tau \leq t \leq T,$}\\
e^{-bc_3\epsilon \xi(t + T - \tau)} & \hbox{if $0 \leq t < \tau \leq T.$}
\end{array}
\right.
\end{align*}
In this point, we introduce the operator $\mathcal{K}_m$ following the previously discussed ideas, and extensively explained in \cite{Far}. We consider the set of functions $U \in C(\RR;\RR^{2m+2})$, such that $U(t)$ has the form 
\begin{align}\label{Ut}
U(t) = (u_0(t),w_0(t),u_1(t),w_1(t),\ldots,u_m(t),w_m(t)).
\end{align}
We define the space
\begin{align*}
C_T(\RR;\RR^{2m+2}) = 
\left\{U \in C(\RR;\RR^{2m+2}): u_i(t + T) = u_i(t), w_i(t + T) = w_i(t), i = \overline{1m},\right\},
\end{align*}
equipped with the norm 
\begin{align}\label{norm1}
\nonumber
\left\|U\right\|_{C_T(\RR;\RR^{2m+2})} &= \sup_{t \in [0,T]}\sqrt{\sum^{m}_{i = 0}\left(u^{(m)}_i(t)\right)^2 + \left(w^{(m)}_i(t)\right)^2} \\
&= \sup_{t \in [0,T]}\sqrt{\left\|u_m(t)\right\|^2_{H} + \left\|w_m(t)\right\|^2_{H}}.
\end{align}
Simultaneously, it is useful to introduce the normed space
\begin{align*}
C_T(\RR;V_m\times H_m) = \left\{(u_m,w_m) \in C(\RR;V_m\times H_m): u_m(t + T) = u_m(t), w_m(t + T) = w_m(t)\right\}
\end{align*}
whose norm is
\begin{align*}
\left\|(u_m,w_m)\right\|_{C_T(\RR;V_m\times H_m)} = \sup_{t \in [0,T]}\sqrt{\left\|u_m(t)\right\|^2_{V} + \left\|w_m(t)\right\|^2_{H}}.
\end{align*}
One can see that there is a continuous isomorphism between the Banach spaces $C_T(\RR;\RR^{2m+2})$ and $C_T(\RR;V_m\times H_m)$, which is defined in the following way: the pair $(u_m, w_m)$ corresponds to each $U$ given by \eqref{resol1}. Next, we will construct the Farkas operator.

We associate the following vector to each $U \in C_T(\RR;\RR^{2m+2})$, we associate the following vector, which depends on $t$ 
\begin{align}\label{Km1}
\left(\mathcal{A}_0\left(U\right)(t),\mathcal{B}_0\left(U\right)(t),\ldots,\mathcal{A}_{m}\left(U\right)(t),\mathcal{B}_m\left(U\right)(t)\right),
\end{align}
where
\begin{align}\label{inteRela1}
\mathcal{A}_i\left(U\right)(t) = \int^T_{0}K_i(t,\tau)\left\langle \psi_i, f(u_m(\tau),w_m(\tau)) + \widetilde{s}(\tau)\varphi^{*}\right\rangle_{V\times V^{*}}d\tau,
\end{align}
and 
\begin{align}\label{inteRela2}
\mathcal{B}_{i}\left(U\right)(t) = \int^T_{0}K_b(t,\tau)u^{(m)}_{i}(\tau)d\tau \;\; \hbox{with $i = 0,\ldots,m,$}
\end{align}
note moreover that operator \eqref{Faroper} for the system \eqref{FGmat} is 
\begin{align*}
\mathcal{K}_m(U) = \left(\mathcal{A}_0\left(U\right)(t),\mathcal{B}_0\left(U\right)(t),\ldots,\mathcal{A}_{m}\left(U\right)(t),\mathcal{B}_m\left(U\right)(t)\right).
\end{align*}
It is easy to see that $\mathcal{K}_m: C_T(\RR;\RR^{2m+2}) \rightarrow C_T(\RR;\RR^{2m+2}).$ Furthermore, if $\mathcal{K}_m$ has a fixed point, then the Faedo-Galerkin system \eqref{FGsys2} has a $T$-periodic solution, see Theorem 4.1.3 in \cite{Far}. 

In what follows, we define
\begin{align}\label{setE}
B^{(m)}_R = \left\{U \in C_T(\RR;\RR^{2m+2}): \sup_{t \in [0,T]}\sqrt{\left\|u_m(t)\right\|^2_{V} + \left\|w_m(t)\right\|^2_{H}} \leq R\right\},
\end{align}
and now, we will look for a fixed point of $\mathcal{K}_m$ into $B^{(m)}_R$. Note that, if $U \in B_R^{(m)}$, the associated pair $(u_m(t),w_m(t))$ satisfies $$\left\|u_m(0)\right\|^2_{H} + \left\|w_m(0)\right\|^2_{H} \leq R_1,$$ for certain constant $R_1 > 0$. Thus, if $U^{(p)}_m \in  B^{(m)}_R$ is a fixed point of $\mathcal{K}_m$, then it is a $T$-periodic solution of the Faedo-Galerkin system \eqref{FGsys2} with initial conditions \eqref{FGintc1} uniformly bounded.

Now, we see that $B^{(m)}_R$ and $\mathcal{K}_m$ satisfy the hypothesis of fixed point theorem of Schauder. We enunciate following result. 
\begin{proposition}\label{TheoPe}
If the parameters of the monodomain model and the constant $R > 0$ satisfy the following relations
\begin{align}\label{ResEq}
\frac{T}{1 - e^{-c_4T}} \leq \frac{\sqrt{2}R}{2\left(1 + \frac{\mathcal{M}}{\alpha}\right)\left(\mathcal{M}_1(R) + \hat{s}\mathcal{N}_{\mathcal{T}}\left\|\varphi\right\|_{L^2(\Gamma_1)}\right)},
\end{align}
\begin{align}\label{ResEq1}
\hbox{and}\;\;\; \xi c_3 \geq \sqrt{2},
\end{align}
where 
\begin{align*}
\mathcal{M}_1(R) = K_1\left(A_1|\Omega|^{3/4} + A_2K_2R^{3} 
+ A_3R^{3/2}\right),
\end{align*}
$A_1, A_2, A_3$ are the constants introduced in Remark \ref{com1}, and $\hat{s} = \sup_{t \in [0,T]}\left|\widetilde{s}(t)\right|.$ Then, $B^{(m)}_R$ is invariant by $\mathcal{K}_m$, that is, $\mathcal{K}_{m}(B^{(m)}_R) \subseteq B^{(m)}_R$. 
\end{proposition}
\begin{proof}
Suppose that $U \in B^{(m)}_R$. Obviously, $\mathcal{K}_m(U) \in C_T(\RR;\RR^{2m+2})$. Now, in order to prove $\mathcal{K}_m(B_R^{(m)})\subseteq B_R^{(m)}$, it is enough to see that
\begin{align*}
\left\|\mathcal{L}_1(u_m,w_m)\right\|^2_V \leq \frac{R^2}{2},\;\;\; \hbox{and}\;\;\; \left\|\mathcal{L}_2(u_m,w_m)\right\|^2_H \leq \frac{R^2}{2}.
\end{align*}
where
\begin{align}\label{L1}
&\mathcal{L}_1(u_m,w_m)(t,\mathbf{x}) = \int^{T}_0\sum^{m}_{i = 0}K_i(t,\tau)\left\langle \psi_i,f(u_{m}(\tau),w_{m}(\tau)) + \widetilde{s}(\tau)\varphi^{*}\right\rangle_{V\times V^{*}}\psi_i(\mathbf{x})d\tau,
\end{align}
and 
\begin{align}\label{L2}
\mathcal{L}_2(u_m,w_m)(t,\mathbf{x}) = \int^{T}_0K_b(t,\tau)u_m(\tau,\mathbf{x})d\tau.
\end{align}
Due to the form of $K_i(t,\tau)$, see \eqref{Ki}, the expression in \eqref{L1} can be written as follows
\begin{align*}
&\mathcal{L}_1(u_m,w_m)(t,\mathbf{x}) = \\
& \int^{t}_0\sum^{m}_{i = 0}(1-e^{-\lambda_i T})^{-1}e^{-\lambda_i(t-\tau)}\left\langle \psi_i,f(u_{m}(\tau),w_{m}(\tau)) + \widetilde{s}(\tau)\varphi^{*}\right\rangle_{V\times V^{*}}\psi_i(\mathbf{x})d\tau + \\
&\int^{T}_t\sum^{m}_{i = 0}(1-e^{-\lambda_i T})^{-1}e^{-\lambda_i(t + T -\tau)}\left\langle \psi_i,f(u_{m}(\tau),w_{m}(\tau)) + \widetilde{s}(\tau)\varphi^{*}\right\rangle_{V\times V^{*}}\psi_i(\mathbf{x})d\tau.
\end{align*}
On the other hand, it is possible to see that 
\begin{align*}
&\sum^{m}_{i = 0}(1-e^{-\lambda_i T})^{-1}e^{-\lambda_i(t-\tau)}\left\langle \psi_i,f(u_{m}(\tau),w_{m}(\tau)) + \widetilde{s}(\tau)\varphi^{*}\right\rangle_{V\times V^{*}}\psi_i = \\
&\left(I_d - e^{-TA}\right)^{-1}e^{-(t - \tau)A}P_m\left(f(u_m(\tau),w_m(\tau)) + \widetilde{s}(\tau)\varphi^{*}\right),
\end{align*}
for all $t > 0$. Therefore
\begin{align*}
&\mathcal{L}_1(u_m,w_m)(t) = \int^{t}_0\left(I_d - e^{-TA}\right)^{-1}e^{-(t - \tau)A}P_m\left(f(u_m(\tau),w_m(\tau)) + \widetilde{s}(\tau)\varphi^{*}\right)d\tau + \\
&\int^{T}_t\left(I_d - e^{-TA}\right)^{-1}e^{-(t + T - \tau)A}P_m\left(f(u_m(\tau),w_m(\tau)) + \widetilde{s}(\tau)\varphi^{*}\right)d\tau.
\end{align*}

Let us estimate $\left\|\mathcal{L}_1(u_m,w_m)(t)\right\|_{V}$. From the Bochner theorem (see \cite{Evans}), we have 
\begin{align*}
&\left\|\mathcal{L}_1(u_m,w_m)(t)\right\|_{V} \leq \\
&\leq \int^{t}_0\left\|\left(I_d - e^{-TA}\right)^{-1}e^{-(t - \tau)A}P_m\left(f(u_m(\tau),w_m(\tau)) + \widetilde{s}(\tau)\varphi^{*}\right)\right\|_{V}d\tau + \\
&\int^{T}_t\left\|\left(I_d - e^{-TA}\right)^{-1}e^{-(t + T - \tau)A}P_m\left(f(u_m(\tau),w_m(\tau)) + \widetilde{s}(\tau)\varphi^{*}\right)\right\|_{V}d\tau.
\end{align*}
Taking into account that
\begin{align*}
&\left\|\left(I_d - e^{-TA}\right)^{-1}e^{-(t - \tau)A}P_m\left(f(u_m(\tau),w_m(\tau)) + \widetilde{s}(\tau)\varphi^{*}\right)\right\|_{V} \leq \\
& \left(1 + \frac{\mathcal{M}}{\alpha}\right)\left\|\left(I_d - e^{-TA}\right)^{-1}\right\|_{\mathcal{L}(H)}\left\|f(u_m(\tau),w_m(\tau)) + \widetilde{s}(\tau)\varphi^{*}\right\|_{V^{*}} \leq\\
& \left(1 + \frac{\mathcal{M}}{\alpha}\right)\left(1 - e^{-\lambda_0T}\right)^{-1}\left\|f(u_m(\tau),w_m(\tau)) + \widetilde{s}(\tau)\varphi^{*}\right\|_{V^{*}},
\end{align*}
for all $t > 0$, we only need to estimate $\left\|f(u_m(\tau),w_m(\tau)) + \widetilde{s}(\tau)\varphi^{*}\right\|_{V^{*}}.$ First, observe that
\begin{align*}
\left\|\widetilde{s}(\tau)\varphi^{*}\right\|_{V^{*}} \leq \hat{s}\mathcal{N}_{\mathcal{T}}\left\|\varphi\right\|_{L^2(\Gamma_1)},
\end{align*}
where we recall that $\hat{s} = \sup_{t \in [0,T]}|\widetilde{s}(t)|$. Also, we have
\begin{align*}
&\left\|f(u_m(\tau),w_m(\tau))\right\|_{V^{*}} \leq K_1 \left\|f(u_m(\tau),w_m(\tau))\right\|_{L^{p'}(\Omega)}\\
& \leq K_1\left(A_1|\Omega|^{1/p'} + A_2\|u_m\|^{p/p'}_{L^p} 
+ A_3\|w_m\|^{2/p'}_{H}\right),
\end{align*}
see Remark \ref{com1}. Here, $K_1$ is the immersion constant from $L^{p'}(\Omega)$ into $V^{*}$. Since, $U \in B^{(m)}_R$ then, the respective pair $(u_m, w_m)$ satisfies 
\begin{align*}
\left\|u_m(t)\right\|_V \leq R,\;\;\; \hbox{and}\;\;\; \left\|w_m(t)\right\|_H \leq R,\;\;\; \hbox{for all $t \in [0,T]$,}
\end{align*}
so, from Remark \ref{com1} it is deduced that $f^N(u_m(t),w_m(t)) \in L^{p'}(\Omega) \subset V^{*}$ and
\begin{align*}
&\left\|f(u_m(t),w_m(t))\right\|_{V^{*}} \leq K_1 \left\|f(u_m(\tau),w_m(\tau))\right\|_{L^{p'}(\Omega)}\\
& \leq K_1\left(A_1|\Omega|^{1/p'} + A_2K_2R^{p/p'} 
+ A_3R^{2/p'}\right),
\end{align*}
where $K_2$ is the immersion constant from $V$ into $L^{p}(\Omega)$. Since, $p = 4$ we obtain the following inequality
\begin{align*}
&\left\|f(u_m(t),w_m(t))\right\|_{V^{*}} \leq K_1\left(A_1|\Omega|^{3/4} + A_2K_2R^{3} 
+ A_3R^{3/2}\right).
\end{align*}

Finally, we arrive at the following estimate 
\begin{align*}
&\left\|\left(I_d - e^{-TA}\right)^{-1}e^{-(t - \tau)A}P_m\left(f(u_m(\tau),w_m(\tau)) + \widetilde{s}(\tau)\varphi^{*}\right)\right\|_{V} \leq \\
& \left(1 + \frac{\mathcal{M}}{\alpha}\right)\left(1 - e^{-\lambda_0T}\right)^{-1}\left(\mathcal{M}_1(R) + \hat{s}\left\|\varphi\right\|_{L^2(\Gamma_1)}\right),
\end{align*}
that is
\begin{align*}
\left\|\mathcal{L}_1(u_m,w_m)(t)\right\|_{V} \leq \left(1 + \frac{\mathcal{M}}{\alpha}\right)\left(1 - e^{-\lambda_0T}\right)^{-1}\left(\mathcal{M}_1 + \hat{s}\left\|\varphi\right\|_{L^2(\Gamma_1)}\right)T,
\end{align*}
for all $t \in [0,T]$ where $\lambda_0 = \frac{\epsilon c_4}{C}$. If the parameters of the monodomain model satisfy
\begin{align}\label{restri1}
\left(1 + \frac{\mathcal{M}}{\alpha}\right)\left(1 - e^{-\frac{\epsilon c_4}{C}T}\right)^{-1}\left(\mathcal{M}_1(R) + \hat{s}\mathcal{N}_{\mathcal{T}}\left\|\varphi\right\|_{L^2(\Gamma_1)}\right)T \leq \frac{\sqrt{2}}{2}R,
\end{align}
then, we obtain
\begin{align*}
\left\|\mathcal{L}_1(u_m,w_m)(t)\right\|^2_{V} \leq \frac{R^2}{2}.
\end{align*}
Note that the inequality \eqref{restri1} is equivalent to the following
\begin{align*}
\frac{T}{1 - e^{-\frac{\epsilon c_4}{C} T}} \leq \frac{\sqrt{2}R}{2\left(1 + \frac{\mathcal{M}}{\alpha}\right)\left(\mathcal{M}_1 + \hat{s}\mathcal{N}_{\mathcal{T}}\left\|\varphi\right\|_{L^2(\Gamma_1)}\right)}.
\end{align*}
On other hand, we have
\begin{align*}
&\left\|\mathcal{L}_2(u_m,w_m)\right\|_H \leq \epsilon b\int^{T}_0K_b(t,\tau)\left\|u_m(\tau)\right\|_{H}d\tau\\
&\leq \epsilon b\int^{T}_0K_b(t,\tau)\left\|u_m(\tau)\right\|_{V}d\tau \leq R\int^{T}_0K_b(t,\tau)d\tau = \frac{R}{\xi c_3}, \;\;\; \hbox{for all $t \in [0,T]$,}
\end{align*}
where we have used the fact that
\begin{align*}
\int^{T}_0K_b(t,\tau)d\tau = \frac{1}{bc_3 \xi \epsilon}.
\end{align*}
Hence, from \eqref{ResEq1} follows 
\begin{align*}
\left\|\mathcal{L}_2(u_m,w_m)\right\|^2_{H} \leq \frac{R^2}{2}.
\end{align*}
So, we have proved that if \eqref{ResEq} and \eqref{ResEq1} are assumed, $U \in B^{(m)}_R$ implies that the associated pair $(u_m,w_m)$ satisfies
\begin{align*}
\sup_{t \in [0,T]}\sqrt{\left\|\mathcal{L}_1(u_m,w_m)(t)\right\|_V + \left\|\mathcal{L}_2(u_m,w_m)(t)\right\|_H} \leq R,
\end{align*}
and thus, $\mathcal{K}_m(B^{(m)}_R) \subseteq B^{(m)}_R.$
\end{proof}

It is easy to see that $B^{(m)}_R$ is a convex, bounded and closed subset of $C_T(\RR;\RR^{2m+1})$. Now, it only remains to be proved that $\mathcal{K}_m$ is a compact operator.
\begin{proposition}
The operator $\mathcal{K}_m: C_T(\RR;\RR^{2m+1}) \rightarrow C_T(\RR;\RR^{2m+1})$ is a compact one. 
\end{proposition}
\begin{proof}
The integral operators defined by \eqref{inteRela1} and \eqref{inteRela2}, have piecewise continuously differentiable kernels $K_i(t,\tau), K_{b}(t,\tau), i = 0,1,\ldots,m,$ in $[0,T]\times [0,T],$ because these only have a jump discontinuity at the points of the form $(t,t)$. On other hand, we claim that each component $\mathcal{A}_i$ and $\mathcal{B}_i$ of $\mathcal{K}_m$ is a compact operator of $C_T(\RR;\RR^{2m+1})$ into $C([0,T];\RR)$. In fact, the operator $\mathcal{A}_i$ is the composition of two operators, $\mathcal{A}^{(1)}_i$ and $\mathcal{A}^{(2)}_i$, where 
\begin{align*}
&\mathcal{A}^{(1)}_i: u \in C([0,T],\RR) \mapsto \int^T_{0}K_{i}(t,\tau)u(\tau)d\tau \in C([0,T],\RR),\;\;\; \hbox{and}\\\\
&\mathcal{A}^{(2)}_i: U \in C_T(\RR;\RR^{2m+1}) \mapsto 
\left\langle \psi_i, f(u_m,w_m) + s(t)\varphi^{*}\right\rangle_{V\times V^{*}} \in C([0,T],\RR).
\end{align*}
Now, for each $i = 0,1,2,\ldots,m,$ the $\mathcal{A}^{(1)}_i$ is compact, see Theorem 1.11 in \cite{Colton}. Here, $C([0,T],\RR)$ is equipped with norm $$\left\|u\right\|_{\infty} = \sup_{t \in [0,T]}|u(t)|.$$

Furthermore, $\mathcal{A}^{(2)}_i$ maps bounded sets of $C_T(\RR;\RR^{2m+1})$ into bounded sets of $C([0,T];\RR)$. In fact, suppose that $U \in C_T(\RR;\RR^{2m+1})$ such that $\left\|U\right\|_{C_T(\RR;\RR^{2m+1})} \leq \widehat{M}$, then, for the respective functions $(u_m, w_m)$, we have 
\begin{align*}
\sqrt{\left\|u_m(t)\right\|^2_H + \left\|w_m(t)\right\|^2_H} \leq \widehat{M}_1,\;\;\; \hbox{for all $t \in [0,T]$},
\end{align*}
hence, making use of the norm equivalence, we have 
\begin{align*}
\sqrt{\left\|u_m(t)\right\|^2_V + \left\|w_m(t)\right\|^2_H} \leq \widehat{M}_2,\;\;\; \hbox{for all $t \in [0,T]$},
\end{align*}
for some constant $\widehat{M}_2 > 0$. Now, 
\begin{align*}
\left|\left\langle \psi_i, f(u_m(t),w_m(t)) + \widetilde{s}(t)\varphi^{*}\right\rangle_{V\times V^{*}}\right| \leq \left(\left\|f(u_m(t),w_m(t))\right\|_{V^{*}} + \hat{s}\mathcal{N}_{\mathcal{T}}\left\|\varphi\right\|_{L^2(\Gamma_1)}\right)\left\|\psi_i\right\|_V.
\end{align*}
Proceeding as in the proof of Proposition \ref{TheoPe}, it is possible to prove
\begin{align*}
&\left\|f(u_m(t),w_m(t))\right\|_{V^{*}} \leq K_1 \left\|f(u_m(\tau),w_m(\tau))\right\|_{L^{p'}(\Omega)}\\
& \leq K_1\left(A_1|\Omega|^{1/p'} + A_2K_2\widehat{M}_2^{p/p'} 
+ A_3\widehat{M}_2^{2/p'}\right),
\end{align*}
and so
\begin{align*}
\sup_{t \in [0,T]}\left|\left\langle \psi_i, f(u_m(t),w_m(t)) + \widetilde{s}(t)\varphi^{*}\right\rangle_{V\times V^{*}}\right| \leq \widehat{M}_3,
\end{align*}
for certain constant $\widehat{M}_3$.

Since, the operator $\mathcal{A}_i$ is the composition of a bounded operator with a compact operator, it follows that it is a compact operator for each $i = 0,\cdots,m$. The proof of the compactness of $\mathcal{B}_i$  is similar hence it will be omitted. In other words, the operator $\mathcal{K}_m: C_T(\RR;\RR^{2m+1})\rightarrow C_T(\RR;\RR^{2m+1})$ is a compact one.
\end{proof}

At this moment, we give the following proposition which is a consequence of the fixed point Schauder theorem and the previous two theorems. 
\begin{theorem}\label{prop2}
Under conditions of Proposition \ref{TheoPe}, the operator $\mathcal{K}_m$ has a fixed point $U^{(p)}_m \in B^{(m)}_R$ which is a $T$-periodic solution of the Faedo-Galerkin system (\ref{FGmat}), for each $m = 0,1,2,\cdots.$ Also, the associated pairs $\left(u^{(p)}_m,w^{(p)}_m\right)$ satisfy that
\begin{align*}
\left\|\left(u^{(p)}_m,w^{(p)}_m\right)\right\|_{C_T(\RR;V_m\times H_m)} \leq R, \;\;\; \hbox{for each $m = 0,1,\cdots.$} 
\end{align*}
\end{theorem}
\subsection{About the convergence of the sequence $\left\{(u_m,w_m)\right\}^{\infty}_{m = 0}$ to a $T$-periodic global weak solution} 
In \cite{nues}, the authors proved that it is sufficient to consider uniformly bounded in $H\times H$ initial conditions, $(u_{0m}, w_{0m}), m = 0,1,\cdots,$ in order that the sequence formed with the solutions $(u_m, w_m), m = 0,1,\cdots,$ of the Cauchy problem \eqref{FGsys}-\eqref{FGintc} has a subsequence which converges to a global weak solution of problem \eqref{ecCn}-\eqref{condFPn}. In Theorem \ref{prop2} of Section \ref{sect1}, we prove that each Faedo-Galerkin system \eqref{FGsys2} has a $T$-periodic solution $U^{(p)}_m \in B^{(m)}_R$, it implies that the associated pairs $\left(u^{(p)}_m, w^{(p)}_m\right)$ are uniformly bounded in $C_T(\RR;V_m\times H_m)$. Hence we have
\begin{align*}
\left\|u^{(p)}_m(0)\right\|_H \leq \widetilde{R},\;\;\; \hbox{and}\;\;\; \left\|w^{(p)}_m(0)\right\|_H \leq \widetilde{R},
\end{align*}
for some constant $\widetilde{R} > 0$. In other words, if we take 
\begin{align*}
u_{0m} = u^{(p)}_m(0),\;\;\; w_{0m} = w^{(p)}_m(0),
\end{align*}
the solution of the corresponding Cauchy problem is $T$-periodic and its initial conditions are uniformly bounded in $H$. The combination of these two results implies that the sequence $\left\{\left(u^{(p)}_m,w^{(p)}_m\right)\right\}^{\infty}_{m = 0}$ has a subsequence which converges to a $T$-periodic global weak solution of problem \eqref{ecCn}-\eqref{condFPn}. This is the main result of our paper.

Because of the role played here by the results obtained in \cite{nues}, we briefly summarize these below, as a complement to this subsection.

We have 
\begin{theorem}\label{imporTheo}
If there is a positive constant $C$ that does not depend on $m$, such that
\begin{align*}
\left\|u_{0m}\right\|_H \leq C, \;\;\; \left\|w_{0m}\right\|_H \leq C,
\end{align*}
then there is a subsequence of $\left\{(u_m,w_m)\right\}^{\infty}_{m = 0}$, formed with solutions of \eqref{FGsys}-\eqref{FGintc}, also denoted by $\left\{(u_m,w_m)\right\}^{\infty}_{m = 0}$, which converges to a global weak solution $(u,w)$. The convergence of this subsequence is such that we have
\begin{align}\label{conv1}
u_m \rightarrow u, \;\;\; \hbox{strongly in $L^2(Q_I),$}\;\;\; w_m \rightarrow w, \;\;\; \hbox{strongly in $L^2(Q_I).$} 
\end{align}
for all interval $I = (0, t_0),$ with $t_0 > 0.$
\end{theorem}
\begin{remark}
From \eqref{conv1} and since that $(u,w)$ is a global weak solution, it turns out that if $u_m$ and $w_m$ are $T$-periodic for all $m$, then $u$ and $w$ are also $T$-periodic a.e.
\end{remark}
Proof of Theorem \ref{imporTheo} can be obtained from the following two propositions.
\begin{proposition}\label{lemma1}
Consider the Cauchy problem \eqref{FGsys}-\eqref{FGintc}, with uniformly bounded initial conditions $(u_{0m}, w_{0m}) \in H\times H$, that is, there is a constant $C_0 > 0$ such that
\begin{align*}
\left\|u_{0m}\right\|_H \leq C_0 \:\: \hbox{and} \:\: \left\|w_{0m}\right\|_H \leq C_0
\end{align*}
for all $m = 0,1,\ldots$ Then, the problem \eqref{FGsys}-\eqref{FGintc} has a solution for all $t > 0$. In addition, there are constants $C_i, i = 1,2,3,4$ such that for all $t_0 > 0$ the following a priori estimates are fulfilled: 
\begin{align*}
&\lambda\left\|u_m(t)\right\|^2_H + \left\|w_m(t)\right\|^2_H \leq C_1,\;\;\; \hbox{for all $t \in [0,t_0]$,}\\
&\left\|u_m\right\|_{L^p(Q_I)\cap L^2(I;V)} \leq C_2,\;\;\; \left\|u'_m\right\|_{L^{p'}(Q_I) + L^2(I;V')} \leq C_3,\;\;\; \left\|w'_m\right\|_{L^{2}(Q_I)} \leq C_4,
\end{align*} 
where $I = (0,t_0), \lambda$ is the positive constant that appears Remark \ref{com1}, and the norms
\begin{align*}
\left\|\cdot\right\|_{L^p(Q_I)\cap L^2(I;V)} &= \max\left\{\left\|\cdot\right\|_{L^p(Q_I)},\left\|\cdot\right\|_{L^2(I;V)}\right\},\\
\left\|u\right\|_{L^{p'}(Q_I) + L^2(I;V')} &= \inf_{u = u_1 + u_2}\left\{\left\|u_1\right\|_{L^{p'}(Q_I)} + \left\|u_2\right\|_{L^2(I;V')}\right\},
\end{align*}
have been used. Here, 
\begin{align*}
u'_m(t) = \sum^m_{i = 0}u^{(m)'}_i(t)\psi_i \in V_m,\;\;\; w'_m(t) = \sum^m_{i = 
0}w^{(m)'}_i(t)\psi_i \in V_m.
\end{align*}
\end{proposition}
Finally, we provide the following proposition which is also proved in \cite{nues}.
\begin{proposition}\label{teo1}
There is a subsequence of $\left\{(u_m,w_m)\right\}^{\infty}_{m = 0}$, denoted by convenience in the same form, such that for all $t_0 > 0$ and $I = (0,t_0)$ satisfy
\begin{align*}
&u_m \rightarrow u,\;\;\; \hbox{weakly in $L^p(Q_I)\cap L^2(I;V)$},\;\;\;u'_m \rightarrow \widetilde{u},\;\;\; \hbox{weakly in $L^{p'}(Q_I) + L^2(I;V')$},\\
&w_m \rightarrow w,\;\;\; \hbox{weakly in $L^2(Q_I)$},\;\;\; w'_m \rightarrow \widetilde{w},\;\;\; \hbox{weakly in $L^2(Q_I)$},\\
\end{align*}
and 
\begin{align*}
u_m \rightarrow u,\;\;\; \hbox{strongly in $L^2(Q_I)$},\;\;\; w_m \rightarrow w,\;\;\; \hbox{strongly in $L^2(Q_I)$}.
\end{align*}
Furthermore, $(u,w)$ is a global weak solution of problem \eqref{ecCn}-\eqref{condFPn}.
\end{proposition}
Proposition \ref{lemma1} means that $u_m, w_m, u'_m$ and $w'_m$ are uniformly bounded in $L^2(I;V), L^2(Q_I), L^{p'}(I;V')$ and $L^2(Q_I)$, respectively. This implies, according to Theorem 5.1 in \cite{Lions}, that there is a subsequence of $\left\{(u_m,w_m)\right\}^{\infty}_{m = 0}$ which converges in $L^2(Q_I)$. 
\subsection{Ionic-diffusive relations leading to the existence of a $T$-periodic weak solution of the monodomain model}
In this section, we give a theorem about the existence of weak $T$-periodic solutions of the monodomain model in the sense of Definition \ref{wealsolution}. The result is a consequence of those shown in the previous section. Specifically, Theorem \ref{prop2} which established the existence of a $T$-periodic and uniformly bounded sequence of Faedo-Galerkin approximations, $\left(u^{(p)}_m,w^{(p)}_m\right) \in C_T(\RR;V_m\times H_m)$, and Theorem \ref{imporTheo} asserting the convergence of these approximations to a $T$-periodic weak solution of \eqref{ecCn}-\eqref{condFPn}. Below, when we talk about ionic-diffusive relations, we refer to relations between the parameters of the monodomain model, considering two classes: the parameters involved in the functions $f_{ion}$ and $g$ called ionic parameters, and the parameters $K_1, K_2$ appearing in Proposition \ref{TheoPe} which are referred by us as the diffusive parameters. 

Observe that out of all these results only Proposition \ref{TheoPe} imposes restrictions on the parameters of the model. The inequalities \eqref{ResEq} and \eqref{ResEq1} set up relations between the ionic parameters, the parameters associated with the diffusion, the geometry of the heart, and the period $T$. Although these are sufficient conditions, they suggest the idea that periodic rhythm is a consequence of the coordination between the process of generation of the action potential at the cellular level, and the propagation of this potential in overall cardiac tissue. 

Before giving the result of existence, we analyze the inequality \eqref{ResEq} in a detailed way, in order to find sufficient conditions for this to be satisfied. These new restrictions of the ionic-diffusive parameters are interpreted in a physiological sense confirming the medical observations in some cardiac pathologies associated with the ionic channels, see \cite{Manlio1,Manlio2,Manlio3}. 

We introduce the following notation
\begin{align*}
\kappa &:= \frac{\sqrt{2}}{2\left(1 + \frac{\mathcal{M}}{\alpha}\right)},\;\;\; \beta := A_2K_1K_2,\;\;\; \gamma := A_3 K_1,\;\;\; \delta := A_1K_1\left|\Omega\right|^{3/4} + \hat{s}\mathcal{N}_{\mathcal{T}}\left\|\varphi\right\|_{L^2(\Gamma_1)},
\end{align*}
where $A_1, A_2$ and $A_3$ depend on the model parameters (see Remark \ref{com1}), $K_1$ and $K_2$ are constants that arise in the proof of Proposition \ref{TheoPe}, and $\mathcal{M}, \alpha$ constitute parameters associated to the bilinear form $a$, see Section \ref{epi1}. On the other hand, if we define  
\begin{align}
\label{HT}h(T) &= h(T,c_4) = \frac{T}{1 - e^{-\frac{\epsilon c_4}{C} T}},\\
\label{PR}p(R) &= p(R,\kappa,\beta,\gamma,\delta) = \frac{\kappa R}{\beta R^3 + \gamma R^{3/2} + \delta},
\end{align}
then inequality \eqref{ResEq} can be written in the form 
\begin{align}\label{R-T}
h(T) \leq p(R).
\end{align}
\begin{remark}
Note that \eqref{R-T} implies an explicit relation of $T$ with $R$, and an implicit relation between the ionic-diffusive parameters $\kappa,\beta,\gamma,\delta,c_4$. The behavior of the functions $ h(T)$ and $ p(R)$ depends on these parameters. As we will see later, it is possible to show that there are finite intervals, which also depend on the ionic-diffusive parameters, $[R_{min}, R_{max}]$ and $[T_{min},T_{max}]$ where \eqref{R-T} is satisfied. In this sense, it can happen $T_{max} - T_{min}$ to be large, for example, for certain values of the ionic-diffusive parameters. If the condition \eqref{R-T} were necessary and sufficient for the existence of a $T$-periodic, then $\left[\dfrac{2\pi}{T_{max}},\dfrac{2\pi}{T_{min}}\right]$ is the interval of the admissible frequencies of the cardiac rhythm, out of this interval, the heart eventually loses its periodic rhythm. Observe also that if $T_{max} - T_{min}$ be large, then $\left[\dfrac{2\pi}{T_{max}},\dfrac{2\pi}{T_{min}}\right]$ is small. Obviously, \eqref{R-T} is not a necessary condition however, the given interpretation agrees what it was observed by the cardiologists. Some channel ionic diseases ("channelopathies") induce to a person to present ventricular fibrillation either with very slow heart rates (i. e. some types of Long QT syndrome and Brugada syndrome) or with fast heart rates (i. e. CPVT), see \cite{Manlio1, Manlio2, Manlio3}. 
\end{remark}

Now, let us give the central theorem. 

\begin{theorem}\label{imptTheo}
Assume the ionic-diffusive parameters satisfy the relations 
\begin{align}\label{temp11}
\xi c_3 \geq \sqrt{2},
\end{align}
and 
\begin{align}\label{temp12}
\frac{\kappa}{\delta} < \frac{C}{\epsilon c_4} < p(R^{*}),
\end{align}
where 
\begin{align}\label{R*}
R^{*} = \left(\frac{\sqrt{\gamma^2 + 32\beta\delta} - \gamma}{8\beta}\right)^{2/3}.
\end{align}
Then, two intervals $[R_{min},R_{max}]$ and $[T_{min},T_{max}]$ can be found, such that, for all $R \in [R_{min}, R_{max}]$ there is a \, $T$-periodic global weak solution $(u,w)$, of the monodomain model \eqref{ecCn}-\eqref{condFPn} with period $T \in [T_{min},T_{max}]$. This periodic solution satisfies the following estimate  
\begin{align*}
\sup_{t \in [0,T]}\sqrt{\left\|u(t)\right\|^2_{V} + \left\|w(t)\right\|^2_{H}} \leq R.
\end{align*}
\end{theorem} 
As we mentioned before, Theorem \ref{imptTheo} is a corollary of previously obtained results. Hence we do not believe necessary to reproduce it here. However, we want to emphasize that \eqref{temp12} is a condition, given in terms of the ionic-diffusive parameters, which ensures \eqref{ResEq} is met. This is easily shown making an analysis of the behavior of functions $h = h(T)$ and $p = p(R)$ given in \eqref{HT} and \eqref{PR}, respectively. It can be seen that the condition \eqref{temp12} means that the minimum value of the function $h(T)$ is between the minimum and maximum values of the function $p(R)$. If this is the case, the graphs of the functions are being cut at two points, as it is seen in Figure \ref{fi1}. The intervals $[R_{min},R_{max}]$ and $[T_{min},T_{max}]$ can be found from the interception points of these curves. In Figure \ref{fi1}, the curves 2 and 3 represent the case when the ionic-diffusive parameters do not satisfy \eqref{temp12}.
\begin{figure}[H]
\centering
\includegraphics[scale=0.8]{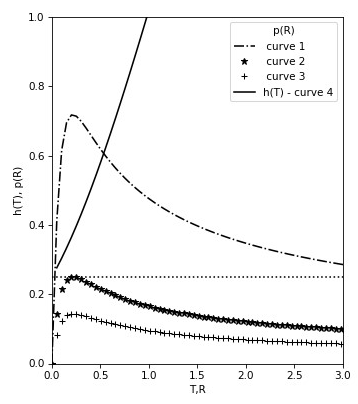}  
\caption{Behavior of the functions $h$ and $p$}
\label{fi1}
\end{figure}
The graphs in Figure \ref{fi1} have been calculated from the parameters given in \cite{Sundnes} for the Roger-McCulloch model, and assuming that $\epsilon, \xi$ are big enough. The parameters $K_1, K_2, \frac{\mathcal{M}}{\alpha}$ have also been estimated in order to obtain $\kappa$ and these are related with the conductivity $\widehat{\sigma}$ and $\Omega$. See Table \ref{tabla1} and Table \ref{tabla2}.
\begin{table}[H]
\begin{center}
\begin{tabular}{|l|l|}
\hline
parameter & value \\
\hline \hline
$\epsilon$ & $0.032$ \\ \hline
$\xi$ & $3.75$ \\ \hline
$\beta$ & $0.0001$ \\ \hline
$\gamma$ & $1$ \\ \hline
$\delta$ & $0.001$ \\ \hline
\end{tabular}
\caption{Table of parameters}
\label{tabla1}
\end{center}
\end{table}
\begin{table}[H]
\begin{center}
\begin{tabular}{|l|l|}
\hline
$\kappa$ & curve \\
\hline \hline
$0.5$ & curve 1 \\ \hline
$0.174$ & curve 2 \\ \hline
$0.1$ & curve 3 \\ \hline
\end{tabular}
\caption{Table of parameters}
\label{tabla2}
\end{center}
\end{table}
\section{Conclusions}
In this paper, we have obtained a result of the existence of weak periodic solutions of the monodomain model for a heart isolated from the torso. In addition to the theoretical-abstract study of the model in order to demonstrate the existence of periodic solutions of the model, one of our fundamental objectives has been to contribute to the understanding of the causes that underlie the generation or loss of heart rhythm. For that, we make an interpretation of the relationships between the ionic and diffusive parameters \eqref{ResEq}-\eqref{ResEq1} given in Proposition \ref{TheoPe}, which turn out to be sufficient for the existence of periodic solutions of period $T$ equal to the activation period of the endocardium. A careful analysis of these relationships, performed in Section 3.3 and summarized in Comment 11, allows us to postulate that the periods with which the heart pulsates are within an interval $[T_1, T_2]$. In other words, we cannot guarantee that there are periodic solutions with periods outside of this interval. From the above, two conclusions can be drawn:
\begin{enumerate}
\item First of all, our heart beats sometimes with high frequencies (when we exercise or get scared), sometimes with low frequencies, for example, when we sleep. However, if for some reason our frequency falls below the minimum frequency determined by the extreme right of the period interval, $ T_2 $, or rises above the maximum value determined by $ T_1 $, then our heart may eventually miss its period rhythm.
\item On the other hand, the size of the period interval $[T_1, T_2]$ depends on the parameters of the model. In other words, it depends on the physiological characteristics of the heart under study. For example, diseases associated with ion channels are reflected in the models with drastic changes in the values of the ionic parameters compared to the values associated with a healthy heart, decreasing the size of the interval $[T_1, T_2]$. For example, a person suffering from Brugada Syndrome may go into fibrillation when sleeping, his or her heart frequency being low but consider to be normal heart frequency in case healthy hearts, see \cite{Manlio2}.
\end{enumerate}
\section*{Acknowledgments}
The authors acknowledge to Consejo Nacional de Ciencia y Tecnología from Mexico the financial support through project A1-S-36879. The authors thankfully acknowledge the computer resources, technical expertise and support provided by the Laboratorio Nacional de Superc\'omputo del Sureste de M\'exico, CONACYT network of national laboratories. The authors also want to thank the English Language graduate Mariana Tamayo for her support in reviewing the language.

\bibliography{mybibfile}

\end{document}